\documentclass[a4paper,11pt]{article}
\usepackage{amsfonts}
\usepackage{graphics}
\newtheorem{theorem}{Theorem}[section]
\newtheorem{lemma}{Lemma}[section]
\newtheorem{definition}{Definition}[section]

\newtheorem{example}{Example}[section]
\newtheorem{remark}{Remark}[section]

\def\dint{\displaystyle \int}

\newcommand{\proof}{{\bf Proof:} }

\newcommand{\eop}{ \hfill $\Box$ }
\begin{document}
\begin{center}
{\Large Existence of optimal controls for SPDE with locally monotone coefficients} 
\footnote{Supported by CAPES Grant 480356/2010-6.} 
\\
\end{center} 

\vspace{0.3cm}

\begin{center}
{\large  Edson A. Coayla-Teran} \\
\textit{Universidade Federal da Bahia-UFBA\\
Av. Ademar de Barros s/n, Instituto de Matem\' atica, Salvador, BA, Brasil, CEP 40170-110, 
e-mail: coayla@ufba.br}\\
{\large  Paulo M. Dias de Magalh\~aes}\\
\textit{Departamento de
Matem\'atica, Universidade Federal de Ouro Preto,\\
Ouro Preto, MG, Brazil,
CEP 35400-000
pmdm@iceb.ufop.br}\\
{\large  Jorge Ferreira}\\
\textit{Departamento de Ci\^encias Exatas, Universidade Federal Fluminense,\\
Av. dos Trabalhadores 420, Vila Santa Cec\'ilia, Volta Redonda, RJ, Brazil, CEP 27255-125\\
ferreirajorge2012@gmail.com}
\end{center}
\begin{abstract}
The aim of this paper is to investigate the existence of optimal controls for systems described by 
stochastic partial differential equations (SPDEs) with locally monotone coefficients controlled by different external forces which are feedback controls. To attain our objective we adapt the argument of \cite{L} where the existence of optimal control to the stochastic Navier-Stokes equation was studied. The results obtained in the present paper may be applied to demonstrate the existence of optimal control to  various types of controlled SPDEs  such as: a stochastic nonlocal equation and stochastic semilinear equations which are locally monotone equations;  we also apply the result to a monotone equation  such as the stochastic reaction diffusion equation  and to a stochastic linear equation.
\end{abstract}

\section{Introduction}	
Let $H$ be a real separable Hilbert space.  Let $V$ be a reflexive Banach space. Identify $H$ with its dual $H'$ and denote the dual of $V$ by $V'$. Let
$$
V\subset H\cong H'\subset V'
$$
where the inclusions are assumed to be dense and compact. The triad $(H,V,V')$ is known as a \textit{Gelfand triple}. We will denote by $\|\cdot\|_V,$ $\|\cdot\|,$ $\|\cdot\|_{V'}$ the norms in $V,$ $H,$ and $V'$ respectively. The inner product in $H$ and the duality scalar product between $V$ and $V'$ will be denoted by $(\cdot,\cdot)$ and $\langle\cdot,\cdot\rangle$ respectively.\\
Let $\left\{W_t\right\}_{t\geq0}$ be a cylindrical Wiener process on a separable Hilbert space $U$ w.r.t. a complete filtered probability space $(\Omega,\mathcal{F},\mathcal{F}_t,\mathbb{P})$ and $(L_2(U;H),\|\cdot\|_2)$ denotes the space of all Hilbert-Schmidt operators from $U$ to $H.$\\
Let $T>0$ be some fixed time. Consider the following initial value problem involving a controlled SPDE of the form: 
\begin{equation}\label{coeq1}
du(t)=(A(t,u(t),u(t))+\Phi(t,u(t)))dt+\Xi(t,u(t))dW(t),\,  u(0)=u_0
\end{equation}
where $A:[0,T]\times V\times V\times\Omega\rightarrow V',$ $\Phi:[0,T]\times H\times\Omega\rightarrow H$ and $\Xi: [0,T]\times V\times\Omega\rightarrow L_2(U;H)$ are progressively measurable, $A$ satisfies a locally monotone condition (see condition {\bf A2} below) and $\Phi$ is a control.\\
In this paper we will study the existence of an optimal control which minimizes the cost function $\mathcal{J}(\Phi)$ with $\Phi$ belonging to $\mathcal{U},$ the set of controls associated with the controlled initial value problem (\ref{coeq1}).\\
\indent The problem of the existence of an optimal control for SPDEs is an important question in 
optimal control theory and often resolved by assuming that the set of admissible controls is 
compact and by using the Main Theorem for Minimum Problems (see \cite{ZE}, Theorem 38.B ).  
In order to answer this question, we use a weaker condition to the set of admissible controls 
which is weak sequentially compact and similarly with the Theorem 38.A of Zeidler \cite{ZE}, we 
assume that the functional cost is weak sequentially lower semicontinuous. The problem of the 
existence of an optimal control for SPDEs  has been studied by several authors, for example, 
by Nagase \cite{N}, Buckdahn and R\u a\c scanu \cite{BR}, Gatarek and Sobczyk \cite{GSS}, 
Guisepina and Federica \cite{GF}, and Al-Hussein \cite{AH} but the results of  these papers 
cannot be applied in the study of the equation in (\ref{coeq1}) because they assume semilinearity
or boundedness for the nonlinearities. As we mentioned previously, the existence of optimal controls for the stochastic Navier - Stokes equation was studied in \cite{L} and we follow the same idea to demonstrate the existence of optimal control to other SPDEs that satisfies a local monotonicity condition. The argument is  to prove that a minimizing sequence has a subsequence which converges weakly (see Lemma \ref{mainlem} ). Then, we prove that weak convergence implies strong convergence of a subsequence of the corresponding solutions, see Theorems \ref{strongconv00} and \ref{maintheo}, these theorems were adapted from \cite{L} to the case of SPDEs with locally monotone coefficients and allow to demonstrate the existence of optimal control to a wide class of SPDEs with locally monotone coefficients as we will see in the examples section. We want to remark that the main result, the Theorem \ref{theomain} of the present work, may also be applied to demonstrate  the existence  of optimal control to locally monotone SPDEs which until the present moment were not studied by other authors.  Specifically, the Examples 3.2 and 3.3 demonstrating the existence of optimal control are new in the literature. \\
\indent The article is organized in the following way: in Section 1, we present the  basic spaces, the norms, properties and notations which we are going to work with in the subsequent sections. In section 2, we formulate the control problem, which is the goal of this work and we prove the existence of an optimal control. Finally, in Section 3 we provide examples where the result of the present paper is applied to some SPEDs such as a nonlocal equation,  semilinear  equation and to other type of SPDEs such as a linear equation and to the stochastic reaction diffusion equation which is a monotone equation.\\
To simplify notation, we use the letter $\mathbb{T}$ for the interval $[0,T]$. Let $(\Omega, \mathcal{F},\mathbb{P})$ be a complete probability space, $\left(\mathcal{F}_t\right)_{t\in\mathbb{T}}$ a right-continuous filtration such that $\mathcal{F}_0$ contains all $\mathcal{F}-$null sets and let $\mathbb{E}(X)$ denote the mathematical expectation of the random variable $X.$ We abbreviate ``\textit{almost surely} $\omega\in\Omega.$" to  a.s.\\
\indent Let $B$ be a Banach space with norm $\|\cdot\|_B$ and let $\mathcal{B}(B)$ denote the Borel $\sigma-$algebra of $B$. The space $L^2(\Omega\times\mathbb{T};B)$ is the set of all $\mathcal{F}\otimes\mathcal{B}(\mathbb{T})-$measurable processes $u:\Omega\times\mathbb{T}\rightarrow B$ which are $\mathcal{F}_t-$ adapted and $\mathbb{E}(\int_{\mathbb{T}}\|u\|^2_Bdt)<\infty.$ The constant $c_{HV}$ is such that $\|v\|^2\leq c_{HV}\|v\|^2_V$ for all $v\in V$.\\
In order to get solutions to (\ref{coeq1}), we state the following conditions on the coefficients:
Suppose there exist constants $\alpha>1,$ $\beta\geq0,$ $\theta>0,$ $K>0$ and a positive adapted process $f\in L^1([0,T]\times\Omega;\mathbb{R})$ such that the following conditions hold for all $v,$ $v_1,$ $v_2\in V$ and a.e. $(t,\omega)\in\mathbb{T}\times\Omega.$
\begin{enumerate}
\item[{\bf A1})]  (Hemicontinuity) The map $s\rightarrow\langle A(t,v_1+sv_2,v_1+sv_2),v\rangle+\langle \Phi(t,v_1+sv_2),v\rangle$ is continuous on $\mathbb{R}.$
\item[{\bf A2})] (Local monotonicity)
$$\begin{array}{rl}
2\langle A(t,v_1,v_1)-A(t,v_2,v_2),v_1-v_2\rangle+&2\langle \Phi(t,v_1)-\Phi(t,v_2),v_1-v_2\rangle+\\
+\|\Xi(t,v_1)-\Xi(t,v_2)\|_2^2\leq&(K+\rho(v_2))\|v_1-v_2\|^2,
\end{array}
$$
where $\rho: V\rightarrow[0,+\infty)$ is a mensurable function and locally bounded in $V.$
\item[{\bf A3})] (Coercivity) 
$$
2\langle A(t,v_1,v_1),v_1\rangle+2\langle \Phi(t,v_1),v_1\rangle+\|\Xi(t,v_1)\|_2^2+\leq-\theta\|v_1\|_V^2+K\|v\|^2+f(t).
$$
\item[{\bf A4})] (Growth) 
$$
\|A(t,v_1,v_1)\|^2_{V'}+\|\Phi(t,v_1)\|^2_{V'}\leq (f(t)+K\|v_1||^2_V)(1+\|v_1\|^{\beta}).
$$
\end{enumerate}
In this work, we understand that the stochastic process $u_{\Phi}$ is a solution to the problem in (\ref{coeq1}) in the following sense.
\begin{definition} Let $u_0$ be a  random variable which does not depend on $W(t).$ The stochastic process $(u_{\Phi}(t))_{t\in\mathbb{T}}\in L^2(\Omega\times\mathbb{T};V),$ $\mathcal{F}_t-$ adapted, with a.s. sample paths continuous in $H$, is a solution to (\ref{coeq1}) if it satisfies the equation:
\begin{equation}\label{iska}\begin{array}{rl}
(u_{\Phi}(t),v)\!=&(u_0,v)+\!\dint_0^t\left\langle A(u_{\Phi}(s),v\right\rangle ds+\dint_0^t(\Phi(s,u_{\Phi}(s)),v)ds+\\
&+\dint_0^t(v,(\Xi(s,u_{\Phi}(s))dW(s))
\end{array}
\end{equation}
a.s. for all $v\in V$ and $t\in\mathbb{T}.$
\end{definition}
\indent {\bf Uniqueness} means indistinguishability.\\
We need the following existence of solutions theorem which is a particular case of Theorem 1.1 of \cite{LR}.
\begin{theorem}\label{exsth} Let $u_0\in L^4(\Omega, V)$. Suppose that (A1) - (A4) is satisfied and there is a constant $C$ such that
\begin{equation}\begin{array}{rl}
\|\Phi(t,v)\|_{V'}^2&+\|\Xi(t,v)\|^2_2\leq C(f(t)+\|v\|^2), \ \ \ t\in\mathbb{T},\ v\in V;\\
\rho(v)\leq &C(1+\|v\|_V^2)(1+\|v\|^{\beta})\ \ \ v\in V.
\end{array}
\end{equation}
The problem (\ref{iska}) has a unique solution $u_{\Phi}$ which  has a.s. sample paths continuous in $H.$ 
\end{theorem}
\begin{proof} See Theorem 1.1 of \cite{LR} .
\end{proof}\eop
\section{Formulation of the control problem and main result}
We consider the SPDE (\ref{coeq1}) controlled by {\it continuous feedback controls} and we denote by $\mathcal{U}:=\left\{\Phi:\mathbb{T}\times L^2(D)\rightarrow L^2(D)\right\}$ the set of the admissible controls satisfying:
\begin{equation}\label{condcopt2}
\|\Phi(0,0)\|^2\leq \eta \ \ \, a.s.
\end{equation}
and for all $s,t\in\mathbb{T},$ $x,y\in H$
\begin{equation}\label{condcopt3}
\|\Phi(t,x)-\Phi(s,y)\|^2\leq \lambda|t-s|^2+\alpha\|x-y\|^2\ \ \,a.s.
\end{equation}
where $\eta, \lambda, \alpha$ are positive constants.\\
Furthermore, we will assume that the coefficients of (\ref{coeq1}) satisfy the following conditions,  for all $v,$ $v_1,$ $v_2\in V$ and a.e.$(t,\omega)\in\mathbb{T}\times\Omega$:
\begin{enumerate}
\item[\bf{C1)}] there is  a constant $L>0$ such that
$$
\|\Xi(t,v_1)-\Xi(t,v_2)\|^2_2\leq L \|v_1-v_2\|^2 \textmd{ and } \|\Xi(0,v_1)\|_2=0
$$ 
\item[\bf{C2)}] there are nonnegative constants $K_1$ and $J_1$ such that
$$
\langle A(t,v,v_1),v_1\rangle\leq -K_1\|v_1\|_V^2+J_1\|v_1\|^2
$$
\item[\bf{C3)}] there is a positive constant $\theta_1$  such that 
$$
\langle A(t,v,v_1)-A(t,v,v_2),v_1-v_2\rangle\leq -\theta_1\,\|v_1-v_2\|_V^2
$$
\item[\bf{C4)}] there are constants $c_1$, $c_2$, $c_3$ and $c_4$ which are nonnegatives and $c_5>0$ such that 
$$\begin{array}{rl}
\langle A(t, v_1,v_2)-A(t,v_3,v_3), &v_2-v_3\rangle\leq-c_5\|v_2-v_3\|_V^2+\\
+(c_4+c_1\rho(v_1))\|v_2-v_3\|^2+&c_2\|v_2-v_1\|^2_V+c_3\rho(v_1)\|v_1-v_2\|^2_V.
\end{array}
$$
\item[\bf{C5)}] there are nonnegative constants $\theta_2,$ $p_3,$  $p_4$ and $p_5$ such that
$$
\| A(t,v,v_1)\|^2_{V'}\leq \theta_2\,\|v_1\|_V^2+p_3\|v\|^2\|v\|^2_V+p_4\|v_1\|^2\|v_1\|^2_V+p_5
$$
\end{enumerate}
\begin{remark} \label{estimt1}Under the conditions (\ref{condcopt2}),  (\ref{condcopt3}) and ({\bf C1}) the solution $u_{\Phi}$ obtained in the Theorem \ref{exsth} satisfies:
\begin{equation}\label{estimat01}
\mathbb{E}(\sup_{t\in\mathbf{T}}\|u_{\Phi}(t)\|^2)+\mathbb{E}(\int_0^T\|u_{\Phi}(s)\|_V^2ds)
\leq c\,\mathbb{E}(\|u_0\|^2)
\end{equation}
and
\begin{equation}\label{estimat02}
\mathbb{E}(\sup_{t\in\mathbf{T}}\|u_{\Phi}(t)\|^4)+\mathbb{E}\left(\dint_0^T\|u_{\Phi}(s)\|^2ds\right)^2
\leq c\,\mathbb{E}(\|u_0\|^4)
\end{equation}
where $c=c(L,\eta,\lambda,\alpha,\theta,T)$ is a positive constant.
\end{remark}
Let us now define the {\it cost functional}
\begin{equation}\label{cost}
\mathcal{J}(\Phi):=\mathbb{E}(\dint_0^T\Bigl(\mathcal{L}(s,u_{\Phi}(s))+\mathcal{K}(\Phi(s,u_{\Phi}(s)))\Bigr)ds)+\mathbb{E}(\mathcal{H}(u_{\Phi}(T))),\ \Phi\in\mathcal{U}
\end{equation}
whenever the integral in (\ref{cost}) exists and is finite, with $\mathcal{L}:\mathbb{T}\times H^1_0(D)\rightarrow\mathbb{R}_+,$\  $\mathcal{K}:L^2(D)\rightarrow\mathbb{R}_+,$ and  $\mathcal{H}:L^2(D)\rightarrow\mathbb{R}_+.$ It is required that the mappings $\mathcal{K},$ $\mathcal{H},$ and $u\in L^2(\mathbb{T};H^1_0(D))\longmapsto\dint_0^T\mathcal{L}(s,u(s))ds$ be weak sequentially lower semicontinuous.\\
\indent Our control problem is to minimize (\ref{cost}) over $\mathcal{U},$ we denote by ($\mathcal{P}$) the problem of minimizing $\mathcal{J}$ among the admissible controls. Any $\Phi^*\in\mathcal{U}$ satisfying $\mathcal{J}(\Phi^*)=\inf\{J(\Phi):\Phi\in\mathcal{U}\}$ is called an {\it optimal control}.\\
The following lemma proves that given a minimizing sequence for the problem ($\mathcal{P}$) we can obtain a subsequence and a mapping $\Phi\in \mathcal{U}$, such that the subsequence converges weakly to $\Phi.$
\begin{lemma}\label{mainlem}Let $\Phi_n$ be a minimizing sequence for problem ($\mathcal{P}$). There exists a subsequence $n_k$ of $n$ and a mapping $\Phi\in\mathcal{U}$ such that for all $t\in\mathbb{T},$ $x,$ $y$ $\in H,$ we have
\begin{equation}\label{weakconv}
\lim_{k\rightarrow\infty}(\Phi_{n_k}(t,x),y)=(\Phi(t,x),y).
\end{equation}
\end{lemma}
\begin{proof} See Lemma 4.1 of  \cite{L}.
\end{proof}\eop\\
\indent For simplicity the subsequence of $\left\{\Phi_{n_k}\right\}_{k=1}^{\infty}$ obtained in the previous lemma  will be relabeled as the same. For this sequence and $\Phi$ as in the last lemma let us consider the equation
\begin{equation}\label{aproxu}\begin{array}{rl}
\vspace{0.1 cm}
(\hat{u}_{\Phi_n}(t),v)&=(u_{0},v)+\dint_0^t \left\langle A(u_{\Phi}(s), \hat{u}_{\Phi_n}(s)),v\right\rangle ds+\\
&+\dint_0^t\left(\Phi_n(s,u_{\Phi}(s)),v\right)ds+\dint_0^t(v,\Xi(s,u_{\Phi}(s))dW(s))
\end{array}
\end{equation}
a.s., $v\in V,$ $t\in\mathbb{T}$ and  for $n\in\mathbb{Z}^+.$ Since the coefficients in the equation (\ref{aproxu}) satisfied the condition ({\bf C2}), ({\bf C3}), ({\bf A1}), ({\bf A3}) and ({\bf A4}), there is a unique process $\hat{u}_{\Phi_n}\in L^2(\Omega\times\mathbb{T};V)$ which is a solution of (\ref{aproxu} ) with a.s. continuous trajectories in $H$  (see  Theorem 4.2.4, p. 75 of \cite{PR} or  Theorem 3.6, p. 32 of \cite{KR1}) satisfying:
\begin{equation}\label{estimat02phi}
\mathbb{E}(\sup_{t\in\mathbb{T}}\|\hat{u}_{\Phi_n}(t)\|^4)+\mathbb{E}\left(\dint_0^T\|\hat{u}_{\Phi_n}(s)\|_1^2ds\right)^2
\leq c\,(\mathbb{E}(\|u_0\|^4)+\mathbb{E}(\dint_0^T\|u_{\Phi}(s)\|^4ds))
\end{equation}
where $c$ is a positive constant independent of $n.$\\
To obtain the estimates  in (\ref{estimat02phi}) we use the Burkholder and Schwarz inequalities.
\begin{theorem} \label{strongconv00}The solution to (\ref{iska}) and  (\ref{aproxu}) satisfies:
$$\lim_{n\rightarrow \infty}\mathbb{E}(\dint_0^T\|(u_{\Phi}-\hat{u}_{\Phi_n})(s)\|_V^2ds)=\lim_{n\rightarrow \infty}\mathbb{E}(\|(u_{\phi}-\hat{u}_{\Phi_n})(T)\|^2)=0.
$$
\end{theorem}
\begin{proof}  Let us consider the equation
\begin{equation}\label{auxequ}
(z(t),v)=(u_0,v)+\int_0^t \left\langle A(u_{\Phi}(s),z(s)),v\right\rangle ds+\dint_0^t(v,\Xi(s,u_{\Phi}(s))dW(s))
\end{equation}
a.s., $v\in V$ and $t\in\mathbb{T}.$  By a similar argument as in the case of equation (\ref{aproxu}), there exists a unique solution $z\in L^2(\Omega\times\mathbb{T}; V)$ of (\ref{auxequ}), which has a.s. continuous trajectories in $H.$ By using the Gronwall lemma, we get the estimate
$$
\mathbb{E}(\sup_{t\in\mathbb{T}}\|z(t)\|^2)+2p\mathbb{E}(\int_0^T\|z(s)\|^2_Vds)\leq k\left(\mathbb{E}(\|u_0\|^2)+\mathbb{E}(\int_0^T\|u_{\Phi}(s)\|^2ds)\right).
$$
Then, there exists $k_2(\omega)>0$ and a.s.,
\begin{equation}\label{estz}\begin{array}{rl}
\sup_{t\in\mathbb{T}}\|z(t)\|^2\leq k_2(\omega),\\
\dint_0^T\|z(s)\|_V^2ds \leq k_2(\omega)
\end{array}
\end{equation}
and
\begin {equation}\label{estiuphi}
\begin{array}{rl}
\sup_{t\in\mathbb{T}}\|u_{\Phi}(t)\|^2\leq k_2(\omega),\\
\dint_0^T\|u_{\Phi}(s)\|_V^2ds \leq k_2(\omega).
\end{array}
\end{equation}
Using stochastic integral properties and (\ref{estimat02}), we obtain that for all $s,t\in\mathbb{T},$ $t>s,$
$$
\mathbb{E}(\|\dint_s^t\Xi(r,u_{\Phi}(r))dW(r)\|^4_{V'})\leq c(t-s)^2E(\|u_0\|^4)
$$
\indent As a result of the Kolmogorov continuity test, we get a random variable $\widetilde{H}$ such that
\begin{equation}\label{estcontin}
\|\dint_s^t\Xi(r,u_{\Phi}(r))dW(r)\|^2_{V'}\leq \widetilde{H}(\omega)|t-s|^{2\gamma}
\end{equation}
a.s. with $0<\gamma<\frac{1}{4}$ and for every $t,s\in\mathbb{T}.$\\
\indent Let $\bar{\Omega}\subset\Omega$ with $\mathbb{P}(\bar{\Omega})=1$ such that for $\omega\in\bar{\Omega}$ the equations in (\ref{iska}) and (\ref{auxequ}) are satisfied and, for each $n\in\mathbb{Z}^+,$ (\ref{aproxu}) is also satisfied and the inequalities in (\ref{estz}), (\ref{estiuphi}) and (\ref{estcontin}) are satisfied.\\
\indent From (\ref{aproxu}), (\ref{auxequ}), (\ref{estiuphi})  and the properties of $A$ ({\bf C3}) and $\Phi_n,$ it follows that for $\omega\in\bar{\Omega},$ 
$$\begin{array}{rl}
\sup_{t\in\mathbb{T}}\|(\hat{u}_{\Phi_n}-z)(t)\|^2+&\theta_1\dint_0^T\|(\hat{u}_{\Phi_n}-z)(s)\|^2_Vds\leq
\frac{c_{HV}^22T(\lambda T^2+\eta)}{\theta_1}+\\
&+\frac{c^2_{HV}2\alpha}{\theta_1}\dint_0^T\|u_{\Phi(s)}\|_V^2ds\leq k(\omega),
\end{array}
$$
where $k(\omega)$ is independent of $n$. Hence, for all $n\in\mathbb{Z}^+,$ we obtain
\begin{equation}\label{estbound}
\sup_{t\in\mathbb{T}}\|\hat{u}_{\Phi_n}(t)\|^2+p\dint_0^T\|\hat{u}_{\Phi_n}(s)\|^2_Vds\leq k(\omega) 
\end{equation}
for $\omega\in\bar{\Omega},$ where $k(\omega)$ is a positive constant independent of $n.$\\
\indent For $\omega\in\bar{\Omega},$ we consider the sequence
$$
F(\omega):=\left\{\hat{u}_{\Phi_n}(\omega,\cdot)\right\}_{n=1}^{\infty},
$$
which is bounded because of (\ref{estbound}).\\
\indent From (\ref{aproxu}),  we obtain
$$\begin{array}{rl}
\|\hat{u}_{\Phi_n}(t)-&\hat{u}_{\Phi_n}(s)\|^2_{V'}\leq \|\dint_s^t\Xi(r,u_{\Phi}(r))dW(r)\|^2_{V'}+\\
&+(t-s)\dint_s^t\left(\|A(u_{\Phi}(r),\hat{u}_{\Phi_n}(r))\|^2_{V'}+\|\Phi_n(r,u_{\Phi}(r))|\|_{V'}^2\right)dr,
\end{array}
$$
for each $t,s\in\mathbb{T},$ $t>s.$ From this, (\ref{estcontin}), (\ref{estbound}) and the properties of $A$ ({\bf C5}), $\Phi_n,$ we get
$$
\|\hat{u}_{\Phi_n}(t)-\hat{u}_{\Phi_n}(s)\|^2_{V'}\leq k(\omega)(t-s)+\widetilde{H}(\omega)(t-s)^{2\gamma}
$$
for $\gamma\in(0,\frac{1}{4})$ and where $k(\omega)>0$ is independent of $n.$\\
\indent Consequently, $F(\omega)$ is equi-continuous in $C([0,T],V')$.  Now, using Dubinsky's Theorem, (see Theorem 4.1, p. 132 of \cite{VF}), it follows that $F(\omega)$ is relatively compact in $L^2(0,T; H).$ Thus, there exists a  subsequence $n_k$ of $n$ and $\hat{u} \in L^2(0,T;H)$ such that
\begin{equation}\label{strongconv}
\lim_{k\rightarrow\infty}\dint_0^T\|(\hat{u}_{\Phi_{n_k}}-\hat{u})(s)\|^2ds=0.
\end{equation}
\indent From  (\ref{aproxu}), (\ref{iska}) and the properties of $A$ ({\bf C3}) we obtain
$$\begin{array}{rl}
\|\hat{u}_{\Phi_{n_k}}(T)-u_{\Phi}(T)\|^2&+2\theta_1\dint_0^T\|(\hat{u}_{\Phi_{n_k}}-u_{\Phi})(t)\|_V^2dt\leq\\\
&\leq\dint_0^T\left(\Phi_{n_k}(t,u_{\Phi}(t))-\Phi(t,u_{\Phi}(t)),(\hat{u}_{\Phi_{n_k}}-\hat{u})(t)\right)dt+\\
&+\dint_0^T\left(\Phi_{n_k}(t,u_{\Phi}(t))-\Phi(t,u_{\Phi}(t)),(\hat{u}-u)(t)\right)dt.
\end{array}
$$
We use Lemma \ref{mainlem}, (\ref{strongconv}) and the properties of $\Phi_n$ and $\Phi$ to obtain
$$
\lim_{k\rightarrow\infty}\|(\hat{u}_{\Phi_{n_k}}-u_{\Phi})(T)\|^2=\lim_{k\rightarrow\infty}\dint_0^T\|(\hat{u}_{\Phi_{n_k}}-u_{\Phi})(t)\|^2_Vdt=0.
$$
\indent Since every subsequence of $(\hat{u}_{\Phi_n}(\omega,\cdot)$ has a subsequence which converges to the same limit $u_{\Phi(\omega,\cdot)}$ in the space $L^2(0,T; V),$ it follows that the sequence $(\hat{u}_{\Phi_n}(\omega,\cdot)$ converges to  $u_{\Phi}(\omega,\cdot)$. Similarly, we can conclude that $(\hat{u}_{\Phi_n}(\omega,T)$ converges to $u_{\Phi}(\omega,T)$ in $H$.\\
\indent From Remark (\ref{estimt1}) and (\ref{estimat02phi}), the processes $\left(\hat{u}_{\Phi_n}\right)_{t\in\mathbb{T}}$ and $\left(u_{\Phi}\right)_{t\in\mathbb{T}}$ are uniformly integrable and thus the theorem follows.
\end{proof}\eop\\
Let $\left(Q(t)\right)$ be a $H^1_0(D)-$valued process with
$$
\int^T_0\|Q(s)\|^2_1ds<\infty\,\textmd{ and }\,\sup_{t\in\mathbb{T}}\|Q(t)\|^2<\infty\ a.s.
$$
For each $M,$ a nonnegative integer, we define the following stopping times:
$$
\bar{\mathcal{T}}_M^{Q}:=\left\{\begin{array}{lr}
\inf\left\{t\in\mathbb{T}:\dint^t_0\|Q(s)\|^2_Vds\geq M\right\},\\
T,\textmd{ if }\dint^T_0\|Q(s)\|^2_Vds<M,&\\
\end{array}
\right.
$$ 
and
$$
\widehat{\mathcal{T}}_M^{Q}:=\left\{\begin{array}{lr}
\inf\left\{t\in\mathbb{T}:\sup_{t\in\mathbb{T}}\|Q(t)\|^2\geq M\right\}\\
T,\textmd{ if }\sup_{t\in\mathbb{T}}\|Q(t))\|^2<M&\\
\end{array}
\right.
$$
and $\mathcal{T}_M^{Q}:=min\left\{\bar{\mathcal{T}}_M^{Q},\widehat{\mathcal{T}}_M^{Q}\right\}.$\\
Let $\Phi_{n}$ and $\Phi$ be the sequence and the map obtained in the Lemma \ref{mainlem}, the following theorem asserts that there is a subsequence $n_k$ of $n$ such that the correspondent  solutions of (\ref{iska}) $u_{\Phi_{n_k}}$ converge strongly to $u_{\Phi}.$
\begin{theorem}\label{maintheo}Let $\left\{\Phi_n\right\}_{n\in\mathbb{N}}$ be as in the last theorem. There is a subsequence $n_k$ of $n$ such that
$$\lim_{k\rightarrow \infty}\mathbb{E}(\dint_0^T\|(u_{\Phi}-u_{\Phi_{n_k}})(s)\|_1^2ds)=\lim_{k\rightarrow \infty}\mathbb{E}(\|(u_{\Phi}-u_{\Phi_{n_k}})(T)\|^2)=0.
$$
\end{theorem}
\begin{proof} For the sake of convenience, we use the abbreviations, $u:=u_{\Phi}$ and $\mathcal{T}_M:= \mathcal{T}_M^u$ for $M=1,2,\dots$.\\
Let $e(t):= exp(\dint_0^t-2c_4-\frac{2\alpha}{c_5}-2L-2c_1\rho(u(s))ds).$ As a result of the It\^o formula, we get
$$\begin{array}{rl}
e(\mathcal{T}_M)\|&\hat{u}_{\Phi_n}(\mathcal{T}_M)-u_{\Phi_n}(\mathcal{T}_M)\|^2=\dint_0^{\mathcal{T}_M}e(s)\|\Xi(s,u(s))-\Xi(s,u_{\Phi_n}(s)\|_2^2ds+\\
&+2\dint_0^{\mathcal{T}_M}e(s)(\Phi_n(s,u(s))-\Phi_n(s,u_{\Phi_n}(s)),(\hat{u}_{\Phi_n}-u_{\Phi_n})(s))ds+\\
&+2\dint_0^{\mathcal{T}_M}e(s)\left\langle A(u(s),\hat{u}_{\Phi_n}(s))-A (u_{\Phi_n}(s),u_{\Phi_n}(s)),(\hat{u}_{\Phi_n}-u_{\Phi_n})(s)\right\rangle ds+\\
&+2\dint_0^{\mathcal{T}_M}e(s)((\hat{u}_{\Phi_n}-u_{\Phi_n})(s),\Xi(s,u(s))-\Xi(s,u_{\Phi_n}(s))dW(s))+\\
&+\dint_0^{\mathcal{T}_M}e'(s)\|(\hat{u}_{\Phi_n}-u_{\Phi_n})(s)\|^2ds.
\end{array}
$$
Using the properties of $A,$ ({\bf C4}), $\Phi_n$ and $\Xi,$ we get
\begin{equation}\label{strongconv01}\begin{array}{rl}
\mathbb{E}(e(\mathcal{T}_M)\|\hat{u}_{\Phi_n}(\mathcal{T}_M)-u_{\Phi_n}(\mathcal{T}_M)\|^2)+c_5&\mathbb{E}(\dint_0^{\mathcal{T}_M}e(s)\|(\hat{u}_{\Phi_n}-u_{\Phi_n})(s)\|^2_V ds)\leq\\
\leq \mathbb{E}(\dint_0^{\mathcal{T}_M}e'(s)\|(\hat{u}_{\Phi_n}-u_{\Phi_n})(s)\|^2ds)+2c_1&\mathbb{E}(\dint_0^Te(s)\rho(u(s))\|(\hat{u}_{\Phi_n}-u_{\Phi_n})(s)\|^2ds)+\\
+2c_2\mathbb{E}(\dint_0^Te(s)\|(\hat{u}_{\Phi_n}-u)(s)\|_V^2ds)+2c_3&\mathbb{E}(\dint_0^Te(s)\rho(u(s))\|(\hat{u}_{\Phi_n}-u)(s)\|_V^2ds)+\\
+\frac{\alpha}{c_5}\mathbb{E}(\dint_0^Te(s)\|(u-u_{\Phi_n})(s)\|^2ds)
+L\mathbb{E}&(\dint_0^{\mathcal{T}_M}e(s)\|(u-u_{\Phi_n})(s)\|^2ds)+\\
+2c_4\mathbb{E}(\dint_0^Te(s)\|(\hat{u}_{\Phi_n}-u_{\Phi_n})(s)\|^2ds)&.
\end{array}
\end{equation} 
From Theorem \ref{strongconv00}, we can get a subsequence  $\left\{\hat{u}_{\Phi_{n_k}}\right\}_{k=1}^{\infty}$ that converges to $ u$ a.e. $(\omega,t)\in\Omega\times\mathbf{T}.$ Thus, from (\ref{strongconv01}), we obtain
$$\begin{array}{rl}
E((e(\mathcal{T}_M)\|\hat{u}_{\Phi_{n_k}}(\mathcal{T}_M)-u_{\Phi_{n_k}}(\mathcal{T}_M)\|^2)+c_5E(\dint_0^{\mathcal{T}_M}e(s)\|(\hat{u}_{\Phi_{n_k}}-u_{\Phi_{n_k}})(s)\|^2_V ds)\leq&\\
\leq 2c_3E(\dint_0^Te(s)\rho(u(s))\|(\hat{u}_{\Phi_{n_k}}-u)(s)\|_V^2ds)+&\\
+(2c_2+2L+\frac{2\alpha}{c_5})E(\dint_0^Te(s)\|(\hat{u}_{\Phi_{n_k}}-u)(s)\|_V^2ds).&\\
\end{array}
$$
From this, Theorems \ref{strongconv00} and  the triangle inequality, we obtain
$$
\lim_{k\rightarrow \infty}E(\dint_0^{\mathcal{T}_M}\|(u_{\Phi}-u_{\Phi_{n_k}})(s)\|_1^2ds)=\lim_{k\rightarrow \infty}E(\|(u_{\Phi}-u_{\Phi_{n_k}})(\mathcal{T}_M)\|^2)=0,
$$
which implies the desired conclusion.
\end{proof}\eop\\
\indent Finally, we are in a position  to formulate our main result.
\begin{theorem}\label{theomain}Under the assumptions of Theorem \ref{exsth}, if, moreover, the conditions ({\bf C1})-({\bf C5})) are satisfied, then there exists an optimal control for the problem $(\mathcal{P})$.
\end{theorem}
\begin{proof} Let $\left\{\Phi_n\right\}$ be a minimizing sequence for the problem ($\mathcal{P}$). We apply Lemma (\ref{mainlem}) and Theorem (\ref{maintheo}) to this sequence. Thus, there exists a subsequence $\left\{\Phi_{n_k}\right\}$ of $\left\{\Phi_n\right\}$ and $\Phi\in\mathcal{U}$ such that, for all $t\in\mathbb{T},$ $x,y\in L^2(D)$ and a.s. $\omega\in\Omega,$ the following hold:
$$
\lim_{k\rightarrow\infty}\left(\Phi_{n_k}(t,u_{\Phi_{n_k}}),y\right)=\left(\Phi(t,u_{\Phi}),y\right)
$$ 
and
$$
\lim_{k\rightarrow\infty}\dint_0^T\|(u_{\Phi_{n_k}}-u_{\phi})(s)\|_V^2ds=\lim_{k\rightarrow\infty}\|(u_{\Phi_{n_k}}-u_{\phi})(T)\|^2=0.
$$
\indent From Theorem (\ref{maintheo}) and the weak sequentially lower semicontinuous properties of $\mathcal{L},$ $\mathcal{K}$ and $\mathcal{H},$ we get
$$
\mathcal{J}(\Phi)\leq\lim \inf_{k\rightarrow\infty}\mathcal{J}(\Phi_{n_k}).
$$
Since $\left\{\Phi_n\right\}$ is a minimizing sequence for the problem ($\mathcal{P}$),
$\mathcal{J}(\Phi)=\min_{\lambda\in\mathcal{U}}\mathcal{J}(\lambda)$ and 
thus $\Phi\in\mathcal{U}$ is an optimal feedback control for problem ($\mathcal{P}$). 
\end{proof}\eop
\section{Examples}
\begin{example}Let $(H,V,V')$ be a Gelfand triple. The main result can be applied to the initial value problem involving the linear stochastic evolution equation:
\begin{equation}\label{monclas}
du(t)=(\mathcal{A}(t,u(t))+\Phi(t,u(t)))dt+\Xi(t,u(t))dW(t), \ \ u(0)=u_0
\end{equation}
where $\mathcal{A}:\mathbb{T}\times V\times\Omega\rightarrow V'$ is a linear operator,  $\Phi$ is the control and $u_0\in H$. Furthermore, we will suppose that there are constants $\alpha_1,$ $\beta_1$ and $\gamma_1$ such that a.e. $(t,\omega)\in\mathbb{T}\times\Omega$ and $v_1,v_2\in V$:
\begin{enumerate}
\item[{\bf  1)}] $|\langle \mathcal{A}(t,v_1),v_2\rangle\leq \alpha_1\|v_1\|_V\|v_2\|_V$
\item[{\bf  2)}]   $|\langle \mathcal{A}(t,v_1),v_1\rangle\leq -\beta_1\|v_1\|^2_V+\gamma_1 \|v_1\|^2.$\\
Then, there is an optimal  control $\Phi$ which minimizes the cost functional $\mathcal{J}$ given by (\ref{cost}).
\end{enumerate}
\end{example}
\begin{proof} Under the conditions (1) , (2) (above), (\ref{condcopt2}), (\ref{condcopt3}) and  ({\bf C1}) it is not hard to prove that the coefficients of the equation  (\ref{monclas}) satisfy the conditions ({\bf A1}), ({\bf A2}) with $K=2\gamma_1+\frac{2L}{\theta_1}+\alpha$ and $\rho(v_2)=0$ and ({\bf A3}) with $\theta=\beta_1,$ $K=\gamma_1$ $f\equiv1$, and ({\bf A4}) with $\beta=2.$ Thus from the Theorem \ref{exsth} there is a solution $u_{\Phi}$ to the equation (\ref{monclas}).\\     
Taking $A(t,u,v)=\mathcal{A}(v)$ we have that the coefficients of the equation  (\ref{monclas}) satisfy ({\bf C2}) with $K_1=\beta_1$, $J_1=0$, ({\bf C3}) with $\theta_1=\beta_1$, ({\bf C4}) with $c_1=c_2=c_3=0,$ $c_5=\beta_1,$ $c_4=\gamma_1$, and ({\bf C5}) with $\theta_2=\alpha_1$, $p_3=p_4=p_5=0$. So that the claim follows from Theorem \ref{theomain}.
\end{proof}\eop\\
The following example shows that the Theorem \ref{theomain} can be applied to some monotone controlled SPDEs. 
\begin{example} (Stochastic Reaction-Diffusion) Let $\mathcal{O}$ be a bounded domain in $\mathbb{R}^d$ with smooth boundary. We can take $H=L^2(\mathcal{O}),$ $V=H^1_0(\mathcal{O})$ and $V'=H^{-1}(\mathcal{O})$ with $p\in[1,+\infty)$ such that $p=\frac{2d}{d-2}$ and $d\geq3.$ Consider the following initial-boundary value problem involving a controlled stochastic reaction diffusion:
\begin{equation}\label{rd}
\left\{\begin{array}{rll}
du(t)=&\!(\Delta u(t)-u(t)|u(t)|^{q-2}+\Phi(t,u))dt+g(t,u)dW(t),\  t\in \left]0,T\right[;\\
u(x,&\!\!\!0)=u_0(x)\textmd{ on }\mathcal{O}\textmd{ and }u(x,t)=0\textmd{ on }\partial \mathcal{O}\times\left]0,T\right[
\end{array}\right.
\end{equation}
where $W$ is a Wiener process in $L^2(\mathcal{O}),$ $q\in[2,p]$ and $\Phi$ is the control.\\
Then, there is an optimal  control $\Phi$ which minimizes the cost functional $\mathcal{J}$ given by (\ref{cost}).
\end{example}
\begin{proof}To prove the claim we use the Theorem \ref{theomain} with:
$$
A(t,u,v)=\Delta v-v|v|^{q-2}.
$$
To demonstrate that $A$ satisfies the conditions ({\bf A1}), ({\bf A2}) (with $K=L$ and $\rho=0$),  ({\bf A3}) and ({\bf A4})  see \cite{PR} example 4.1.5 so that from the Theorem\ref{exsth} there is a unique solution for the equation (\ref{rd}). 
On the other hand, $A$ satisfies the condition ({\bf C2}) with $K_1=1.$ In fact
$$
\langle A(t,v,v_1),v_1\rangle=\langle \Delta v_1-v_1|v_1|^{q-2},v_1\rangle\leq-\|v_1\|^2_V.
$$
To demonstrate that $A$ satisfies the condition ({\bf C3}) with $\theta_1=1,$ we observe that
$$\begin{array}{rl}
\langle A(t,v,v_1)-A(t,v,v_2),v_1-v_2\rangle&=\langle \Delta (v_1-v_2),v_1-v_2\rangle+\\
&-\langle v_1|v_1|^{q-2}-v_2|v_2|^{q-2},v_1-v_2\rangle\leq\\
&\leq-\|v_1-v_2\|_V^2
\end{array}
$$
because the map $u\mapsto -u|u|^{q-2}$ satisfies a local monotonicity condition with $L=0$ and $\rho=0.$\\
Analogously, we can prove that the condition ({\bf C4}) is satisfied with $c_5=1,$ $c_1=c_3=0$ and $c_2=c_4=1$, in a similar manner we can demonstrate that ({\bf C5}) is satisfied to suitable constants.
\end{proof}\eop\\
\begin{remark}We wish to remark that although the equation (\ref{rd}) is well known, this is the first time that the problem of the existence of optimal control to this equation is studied. 
\end{remark}
Now we will consider the following initial-boundary value problem involving a controlled SPDE:
\begin{equation}\label{ksa}
\left\{\begin{array}{rll}
du(t)=&\!(a(\int_Dudx)\Delta u+\Phi(t,u))dt+g(t,u)dW(t)\textmd{ on } t\in \left]0,T\right[,\\
u(x,&\!\!\!0)=u_0(x)\textmd{ on }D\textmd{ and }u(x,t)=0\textmd{ on }\partial D\times\left]0,T\right[
\end{array}\right.
\end{equation}
where $D$ is a bounded open subset of $\mathbb{R}^n$ with  smooth boundary $\partial D,$  $n\geq1,$ $a=a(s)$ is a continuous function with Lipschitz constant $L$ such that $0<p\leq a(s)\leq P$ where $p$ and $P$ are constants, $W$ is a Wiener process in $L^2(D)$ and $\Phi\in\mathcal{U}$ is a control.\\
In this case the Gelfand triple
$$
V\subset H=H'\subset V',
$$
where $V=H^1_0(D)$ and $H=L^2(D).$ 
\begin{example} (Stochastic nonlocal parabolic equation) There is an optimal  control $\Phi$ which minimizes the cost functional $\mathcal{J}$ given by (\ref{cost}) to the equation (\ref{ksa}).
\end{example}
\proof In this example we will consider $A(t,u,v)=(a(\int_Dudx)\Delta v$ for $t\in\mathbb{T},$ $u,v\in V.$ First, we will verify that if $u_0\in L^4(\Omega, H)$ then (\ref{ksa}) has a unique solution $u=u_{\Phi}.$ In fact, the hemicontinuity ({\bf A1}) is a consequence of the properties of $a$.\\
About ({\bf A2}), we have
$$\begin{array}{rl}
\langle A(t,u,u)-A(t,v,v),u-v\rangle\leq-\langle a(\int_Du(x)dx)&\nabla u-a(\int_Dv(x)dx)\nabla v,\nabla(u-v)\rangle\leq\\
\leq-\langle a(\int_Du(t,x)dx)(\nabla u-&\nabla v),\nabla(u-v)\rangle+\\
-\langle (a(\int_Du(t,x)dx)-&a(\int_Dv(t,x)dx))\nabla v,\nabla(u-v)\rangle,
\end{array}
$$
then
$$
\langle A(t,u,u)-A(t,v,v),u-v\rangle+ \frac{p}{2}\|\nabla(u-v)\|^2\leq\frac{C(D)L_1}{2p}\|u-v\|^2\|v\|_V^2,\\
$$
thus
$$\begin{array}{rll}
2(A(t,u)-A(t,v),u-v\rangle+&2\langle\Phi(t,u)-\Phi(t,v),u-v\rangle+&\|\Xi(t,u)-\Xi(t,v)\|_2^2\leq\\
\leq\frac{C(D)L_1}{p}\|v\|^2_V\|v-u\|^2&+L\|v-u\|^2+\frac{2\alpha}{p}\|v-u\|^2.\\
\end{array} 
$$
Hence, we have the local monoticity ({\bf A2}) with $\rho(v)=\frac{C(D)L_1}{p}\|v\|^2_V$ where $C(D)=1_D$. \\
We proceed to demonstrate ({\bf A4}), we have
$$
|\langle A(t,u),w\rangle|^2\leq P\|u\|^2_V\textmd{ for } \|w\|_V\leq1
$$
so we have ({\bf A4}) with $\beta=2$. Similarly, ({\bf A3}) is verified.
Thus, from Theorem (\ref{exsth}) there is a unique solution for the equation (\ref{ksa}).\\
The properties of $a$ provide ({\bf C2}) with $K_1=p$ and $J_1=0$.  Using the properties of $a$ we obtain ({\bf C3}) with $\theta_1=p$. Now, we proceed to demonstrate ({\bf C4}.)   In fact, since
$$
\begin{array}{rl}
\langle A(v_1,v_2&)-A(v_3,v_3),(v_2-v_3)\rangle=\\
&=\left\langle A(v_1,v_2)-A(v_3,v_2),(v_2-v_3)\right\rangle+\\
&+\left\langle A(v_3,v_2-v_3),(v_2-v_3)\right\rangle=\\
&\!\!\!-\left((a(\dint_{D}v_1(x)dx-a(\dint_Dv_3(x)dx))\nabla v_2-\nabla v_1,\nabla(v_2-v_3)\right)+\\
&-\left((a(\dint_{D}v_1(x)dx-a(\dint_Dv_3(x)dx))\nabla v_1,\nabla(v_2-v_3)\right)+\\
&+\left\langle A(v_3,v_2-v_3),(v_2-v_3)\right\rangle\leq\\
&\leq
-\frac{3p}{4}\|v_2-v_3\|^2_V+\frac{2P}{p}\|v_2-v_1\|^2_V+\\
&+\frac{4L_1^2C(D)}{p}\|v_1\|^2_V\|v_2-v_3\|^2+\frac{4L_1^2C(D)}{p}\|v_1\|^2_V\|v_1-v_2\|^2_V
\end{array}
$$
thus ({\bf C4}) is satisfied with $c_1=4L_1,$ $c_2=\frac{2P}{p},$ $c_3=4L_1,$ $c_5=\frac{3p}{4}$ and $c_4=0.$
Using the properties of $a$ we obtain ({\bf C5}) with $\theta_2=P$ and $p_3=p_4=p_5=0,$ 
and the claim follows from Theorem \ref{theomain}. 
\eop\\
\begin{remark}We wish to remark that the the problem of the existence of optimal control to the equation (\ref{ksa}) showing that the equation (\ref{ksa}) is a particular case of a monotone locally equation is studied for the first time in the present work and for this reason we need to demonstrate that in fact the coefficients satisfy the conditions  ({\bf A1})-({\bf A4}).   
\end{remark}
Let $D\subset \mathbb{R}^n$ be an open bounded domain with smooth boundary. 
\begin{lemma}\label{lemmsilin}Consider the Gelfand triple
$$
V:=H^1_0(D)\subset H:=L^2(D)\subset V':=H^{-1}(D)
$$
and the operator
$$
A(u)=\Delta u +\sum_{i=1}^{d} f_i(u)D_i u,
$$
where $f_i,$ for $i=1,\ldots,d$ are bounded Lipschitz functions on $\mathbb{R}.$\\
(1) If $d<3,$ there exists a constant $K_2$ such that
\begin{equation}\label{lady}
2\langle A(u)-A(v),u-v\rangle\leq-\|u-v\|^2_V+(K_2+K_2\|v\|_V^2)\|u-v\|^2_H, \, u,\, v\in V.
\end{equation}
(2) If $d=3,$ there exists a constant $K_3$ such that
$$
2\langle A(u)-A(v),u-v\rangle\leq-\|u-v\|^2_V+(K_3+K_3\|v\|_V^4)\|u-v\|^2_H, \, u,\, v\in V.
$$
(3) If $f_i$ are independent of $u$ for $i=1,\ldots,d$ , i.e.
$$
A(u)=\Delta u +\sum_{i=1}^{d} f_iD_i u,
$$
then for $d\geq1$ we have
$$
2\langle A(u)-A(v),u-v\rangle\leq-\|u-v\|^2_V+K_4\|u-v\|^2_H, \, u,\, v\in V.
$$
where $K_4$ is a constant.
\end{lemma}
\proof See Lemma 3.1 of \cite{LR}\eop
\begin{example}( Stochastic semi-linear equations). Let $d\leq3$ and consider the initial value problem involving the controlled semi-linear stochastic equation
\begin{equation}\label{smln}
du(t)=(\Delta u(t) +\sum_{i=1}^{d} f_i(u(t))D_i u(t)+\Phi(u(t)))dt+\Xi(u(t))dW(t),\ \ u(0)=u_0
\end{equation}
where $W(t)$ is a Wiener process on $L^2(D)$, $\Phi$ is the control and $f_i$ are bounded Lipschitz functions on $\mathbb{R}$ for $i=1,\ldots,d$. Suppose that $|f_i(x)|\leq J<1.$ There is an optimal  control $\Phi$ for the problem ($\mathcal{P}$).
\end{example} 
\proof We can suppose that all $f_i$ with $i=1,\ldots,d$ have the same Lipschitz constant $L_1$.\\
We define the map
$$
A(u,v)=\Delta v+\sum_{i=1}^{d} f_i(u)D_i v, \ \ u\in V.
$$
The hemicontinuity ({\bf A1}) follows from the continuity of $f$ and $\Xi$. We give the proof of ({\bf A2})-({\bf A4}) only for the case $d=3$; the  case $1\leq d<3$ is similar.\\
Therefore, by Lemma \ref{lemmsilin}
$$
2\langle A(u)-A(v),u-v\rangle+2\langle \Phi(u)-\Phi(v),u-v\rangle\leq -\frac{1}{2}\|u-v\|^2_V+({\alpha}{2}+K_2\|v\|_V^4)\|u-v\|^2_H
$$ 
for $u,v\in V.$ Hence,  ({\bf A2}) and ({\bf A3}) are satisfied with $\alpha=2$ and $\rho(v)=K_2\|v\|_V^4$ respectively.
We proceed to demonstrate ({\bf A4}) since we have that
$$
|\langle Au,v\rangle|^2\leq C\|u\|^2_V\|v\|^2_V
$$
so we have ({\bf A4}) with $\beta=2.$ Thus, from Theorem (\ref{exsth}) there is a unique solution for the equation (\ref{smln}).\\
Now we proceed to verify  ({\bf C2}) - ({\bf C5}). Using the properties of $f_i$ we have that ({\bf C2}) is satisfied with $K_1=1-J$ and $J_1=0$.  Since 
$$
\langle A(v,v_1)-A(v,v_2),v_1-v_2)\rangle\leq-(1-J)\|v_1-v_2\|^2_V
$$
we have that ({\bf C3}) is satisfied with $\theta_1=1-J$.\\
Using the properties of $f_i$ we obtain the following inequality
\begin{equation}\label{inqsln}\begin{array}{rl}
\langle A(v_1,v_2)-A(v_3,v_3),v_2-v_3)\rangle\leq-\|v_2-v_3\|_V^2+&\\
\sum_{i=1}^d\int_D (f_i(v_1)-f_i(v_3))(D_iv_2-D_iv_1)(v_2-v_3)dx+&\\
+\sum_{i=1}^d\int_D(f_i(v_1)-f_i(v_2))D_i(v_1)(v_2-v_3)dx+&\\
+\sum_{i=1}^d\int_D(f_i(v_2)-f_i(v_3))D_i(v_1)(v_2-v_3)dx+&\\
+\sum_{i=1}^d\int_Df_i(v_3)(D_iv_2-D_iv_3)(v_2-v_3)dx\leq&\\
\leq-\|v_2-v_3\|_V^2+2J\|v_2-v_3\|\|v_2-v_1\|_V+&\\
+L_1\|v_1-v_2\|_{L^4(D)}\|v_2-v_3\|_{L^4(D)}\|v_1\|_V+&\\
+L_1\|v_2-v_3\|^2_{L^4(D)}\|v_1\|_V+J\|v_2-v_3\|\|v_2-v_3\|_V, 
\end{array}
\end{equation}
for $v_1,v_2,v_3\in V.$ For $d<3$, from inequality (\ref{lady}) and from (\ref{inqsln}) we have
$$\begin{array}{rl}
\langle A(v_1,v_2)-A(v_3,v_3),v_2-v_3)\rangle\leq-\frac{1}{4}\|v_2-v_3\|_V^2+&\\
+\|v_2-v_3\|^2\|v_1\|^2+2J^2\|v_2-v_3\|^2+(1+\frac{L_1^4}{2})\|v_2-v_1\|_V^2+&\\
+\frac{1}{2}\|v_1-v_2\|^2_V\|v_1\|^2_V+
\end{array}
$$
thus ({\bf C4}) is satisfied with $c_1=\frac{1}{K_2},$ $c_2=(1+\frac{L_1^4}{2}),$ $c_3=\frac{1}{2K_2},$  $c_5=\frac{1}{4}$ and $c_4=2J^2$.\\
For $d=3$, from (\ref{inqsln}), Young's inequality and the following inequality (see p. 34 of \cite{MS}):\\
$$
\|u\|^4_{L^4(D)}\leq4\leq\|u\|_{L^2(D)}\|\nabla u\|^3_{L^2(D)}\ \ \ u\in V
$$
we get
$$
\begin{array}{rl}
\langle A(v_1,v_2)-A(v_3,v_3),v_2-v_3)\rangle\leq-\frac{1}{4}\|v_2-v_3\|_V^2+&\\
+L_1^43^3(\frac{1+2^8}{2^6})\|v_2-v_3\|^2\|v_1\|^4+2J^2\|v_2-v_3\|^2+(1+\frac{3L_1^4}{4})\|v_2-v_1\|_V^2+&\\
+\frac{1}{4}\|v_1-v_2\|^2_V\|v_1\|^4_V+\\
\end{array}
$$
thus ({\bf C4}) is satisfied with $c_1=\frac{1}{K_3}L_1^43^3(\frac{1+2^8}{2^6}),$ $c_2=1+\frac{3L_1^4}{4},$ $c_3=\frac{1}{4K_3},$  $c_5=\frac{1}{4}$ and $c_4=2J^2$.\\
Finally, ({\bf C5}) is satisfied with $\theta_2= 1+L,$ $p_3=p_4=0$ and $p_5=1,$    
and the claim follows from Theorem \ref{theomain}. 
\eop

\end{document}